\documentclass[12pt]{article}
\usepackage[a4paper,left=25mm,right=25mm,head=20mm,bottom=20mm]{geometry}
\usepackage[english]{babel}
\usepackage[T2A]{fontenc}
\usepackage{amsmath,amsfonts,amssymb, amsthm}
\usepackage{cite}
\usepackage{xcolor}

\usepackage{tikz}

\theoremstyle{plain}
\newtheorem{theorem}{Theorem}
\newtheorem{lemma}{Lemma}
\newtheorem{corollary}{Corollary}[theorem]

\theoremstyle{definition}
\newtheorem{definition}{Definition}
\newtheorem{note}{Note}
\newtheorem{example}{Example}

\begin{document}


\author{Monakhov~V.\,S., Sokhor~I.\,L. }

\title{On strictly $2$-maximal subgroups of finite groups}

\date{}

\maketitle

\section{Introduction}

All groups in this paper are finite.
We write $H\leq G$ ($H< G$) if $H$ is
a (proper) subgroup of a group $G$.
A subgroup~$M$ of a group~$G$ is called a maximal subgroup
if~$M<G$ and $M\le H\le G$ implies that either~$M=H$ or~$H=G$.
If $M$ is a maximal subgroup of a group~$G$,
then we write $M\lessdot \,G$.

Let $G$ be a group and let $H$ be a subgroup of $G$.
We use the following notation
\[
\mbox {Max} (G,H)=\{M\lessdot \,G\mid H\le M\}.
\]
If $H=1$ is the unit subgroup of~$G$,
then we write $\mbox {Max} (G)$ instead of $\mbox {Max} (G,1)$.
It is clear that $\mbox {Max} (G)$ is the set of all maximal subgroups of~$G$.
It should be noted that $\mbox {Max} (G)=\varnothing$ exactly when~$G=1$.

\begin{definition}
A subgroup $H$ of a group~$G$ is called

a {\sl $2$\nobreakdash-\hspace{0pt}maximal subgroup} of~$G$
if there is $M\in \mbox {Max} (G,H)$ such that $H\lessdot \,M$;

an {\sl $n$-maximal subgroup} of~$G$ for $n\ge 3$
if there is $M\in \mbox {Max} (G,H)$ such that $H$ is
an $(n-1)$-maximal subgroup in~$M$.
\end{definition}

\begin{example}
In~$L_2(8)$ ~\textup{\cite[IdGroup(504,156)]{gap}},
a subgroup~$C_2$ is a $2$-, $3$- and $4$-maximal subgroup:

$C_2\lessdot \, D_{14}\lessdot \, L_2(8)$,

$C_2\lessdot \, S_3\lessdot \,D_{18}\lessdot \,L_2(8)$,

$C_2\lessdot \, C_2^2\lessdot \,C_2^3\lessdot \,C_2^3:C_7 \lessdot \, L_2(8)$.
\end{example}

In view of~\cite[example 3]{mkRic}, for any~$n>2$
there is a group in which a $2$-maximal subgroup
is $n$-maximal.

\begin{definition}
A subgroup $H$ of a group~$G$ is called
a {\sl strictly $2$\nobreakdash-\hspace{0pt}maximal subgroup} of~$G$
if $H\lessdot \,M$ for all $M\in \mbox {Max} (G,H)$.
Clearly, a strictly 2-maximal  subgroup of a group~$G$
is 2-maximal in~$G$ and is not $n$-maximal in~$G$ for any~$n>2$.
\end{definition}

By $\mbox {Max}_2(G)$ we denote the set
of all 2-maximal subgroups of a group~$G$,
$\mbox {Max}^\star_2(G)$ denotes the set of all
strictly 2-maximal subgroups of~$G$.
It is clear that $\mbox {Max}_2(G)=\varnothing$ exactly
when~$G=1$ or~$|G|$ is a prime. From the indices lemma,
it follows that a 2-maximal subgroup of least index
is strictly 2-maximal, see Lemma~\ref{l1}.
Therefore $\mbox {Max}_2^\star(G)\ne \varnothing$ for
any $G\neq 1$ of nonprime order.

The first author of this paper proper the following problem~\cite[19.54]{kour}

{\sl What are the chief factors of a finite group in which every
2-maximal subgroup is not $n$-maximal for any $n\ge 3$?}

This problem is researched in~\cite{mg}.

If every $2$-maximal subgroup of a group~$G$ is not
$n$-maximal for all~$n\ge 3$, then
$\mbox {Max}_2(G)=\mbox {Max}_2^\star (G)$,
i.\,e. every $2$-maximal subgroup of~$G$ is strictly 2-maximal.
Hence the noted problem could be formulated  as follows.

{\sl What are the chief factors of a finite group in which
$\mbox {Max}_2(G)=\mbox {Max}_2^\star (G)$?}

The examples of groups with $\mbox {Max}_2(G)=\mbox {Max}_2^\star (G)$
are supersoluble groups,
the nonsupersoluble group $C_3^2: C_8$ \cite[remak 4]{mg}, 
the group~$U_3(2)$, the simple groups~$U_3(3)$ and~$L_2(17)$,
see examples~\ref{ef9}--\ref{eu33}.

In this paper,  we conclude from the results of
Hanguang Meng and Xiuyun Guo~\cite{mg} some corollaries
about the existence of strictly $2$-maximal subgroups in groups.
We give examples of groups that illustrate properties of strictly $2$-maximal subgroups.

\section{On groups with $\mbox {Max}_2(G)=\mbox {Max}_2^\star(G)$}

\begin{lemma}\label{l1}
If $G\neq 1$ is a group of nonprime order,
then $\mbox {Max}_2^\star(G)\ne \varnothing$.
\end{lemma}

\begin{proof}
Let $H$ be a $2$-maximal subgroup in $G$ of least index.
Suppose that $H$ is not a strictly $2$-maximal subgroup.
Then there  is $M\in \mbox {Max} (G,H)$ such that
$H$ is not a maximal subgroup in~$M$. So, in~$M$ there is
a subgroup~$K$ such that $H<K\lessdot \,M$.
By the indices lemma,
\[
|G:H|=|G:K||K:H|, \ |K:H|\ne 1, \ |G:K|<|G:H|.
\]
Thus, $K$ is $2$-maximal in~$G$ and $|G:K|<|G:H|$,
this contradicts the choice of $H$.
Hence we conclude that $H$ is a strictly $2$-maximal subgroup of~$G$.
\end{proof}

\begin{lemma}\label{lem110}
Let $H$ be a $2$-maximal subgroup of a group~$G$,
$H\lessdot \, M\lessdot \,G$. If the indices $|G:M|$ and~$|M:H|$
are primes, then $H$ is a strictly 2-maximal subgroup of~$G$.
In particular, if~$G$ is a supersoluble group,
then $\mbox {Max}_2(G)=\mbox {Max}_2^\star(G)$.
\end{lemma}

\begin{proof}
Assume that $H$ is a $2$-maximal subgroup of~$G$,
$H\lessdot \, M\lessdot \,G$, and the indices~$|G:M|$ and~$|M:H|$
are primes. Suppose that $H$ is not a strictly 2-maximal
subgroup of~$G$. Hence there is a  subgroup~$K$ of~$G$
such that~$H<K<G$ and~$H$ is 2-maximal in~$K$.
Therefore there is a subgroup~$L$ such that~$H\lessdot \,L\lessdot \,K<G$.
By the indices lemma,
\[
|G:H|=|G:K||K:L||L:H|, \ |G:K|\ne 1, \ |K:L|\ne 1, \ |L:H|\ne 1,
\]
so $|G:H|$ is divided by three primes, a contradiction.
Consequently, $H$ is a strictly $2$-maximal subgroup of~$G$.

Let $H\lessdot \,M\lessdot \,G$ and let $G$ be a supersoluble group.
By the Huppert Theorem~\cite[VI.9.5]{hup}, $|G:H|$ is divided by
exactly two not necessarily different primes.
If $H<X\lessdot \,G$, then $|X:H|$ is a prime
and~$H$ is a  maximal subgroup in~$X$.
Since~$X$ is an arbitrary maximal subgroup of~$G$ containing $H$,
we obtain that~$H$ is a strictly 2-maximal subgroup of~$G$.
\end{proof}

We give examples nonsupersoluble groups with
$\mbox {Max}_2(G)=\mbox {Max}_2^\star(G)$.
In examples, we based on~\cite{gap,cl,atlas,GN}
and build a graph for each group,
whose vertices are representatives of the classes
of conjugate subgroups and two vertices $A$ and $B$
are joined by an edge whenever $B\lessdot A$,
at that $B$ is located below $A$.
We follow the notation of~\cite{atlas}.
Besides, $C_q$ denotes a cyclic group of order $q$,
$G_q^n$ denotes a direct product of $n$ copies of $C_q$.

\begin{example}[{\cite[remak 4]{mg}}]\label{ef9}
$C_3^2: C_8$~(\cite[IdGroup(72,39)]{gap}, \cite{GN})

\begin{center}
\begin{tikzpicture}
[level distance=10mm, draw=white,
level 1/.style={sibling distance=6em, level distance=10mm},
level 2/.style={sibling distance=4em, level distance=10mm},
level 3/.style={sibling distance=4em, level distance=10mm},
level 4/.style={level distance=10mm},
]
\node(0) {$C_3^2: C_8$}
    child{node(11) {$C_3^2: C_4$}
          child{node(21) {$C_3: S_3$}
                child[missing]{}
                child{node(31) {$C_3^2$}}
                child{node(32) {$S_3$}
                    child{node(41) {$C_3$}
                        child[missing]{}
                        child{node(5) {$1$}}}
                    child{node(42) {$C_2$}}}}}
    child{node(12) {$C_8$}
          child{node(22) {$C_4$}}};
    \draw[gray] (0)--(11) (0)--(12);
    \draw[gray] (11)--(21) (11)--(22) (12)--(22);
    \draw[gray] (21)--(31) (21)--(32);
    \draw[gray] (31)--(41) (32)--(41) (32)--(42) (22)--(42);
    \draw[gray] (41)--(5) (42)--(5);
\end{tikzpicture}
\end{center}

$\mbox {Max}(C_3^2: C_8)= \{C_8,C_3^2: C_4\}$,

$\mbox {Max}_2(C_3^2: C_8)=\{C_4,C_3: S_3\}$,

$\mbox {Max}^\star _2(C_3^2: C_8)=\mbox
{Max}_2(C_3^2: C_8)$.
\end{example}

\begin{example}\label{epsu32}
$U_3(2)$~(\cite[IdGroup(72,41)]{gap}, \cite{GN})

\begin{center}
\begin{tikzpicture}
[level distance=10mm, draw=white,
level 1/.style={sibling distance=5em, level distance=10mm},
level 2/.style={sibling distance=3em, level distance=14mm},
level 3/.style={sibling distance=3em, level distance=10mm},
level 4/.style={level distance=10mm},
level 5/.style={level distance=10mm},
]
\node {$U_3(2)$}
    child{node(11) {$C_3^2: C_4$}
          child{node(21) {$C_3: S_3$}
                child[missing]{}
                child{node(31) {$C_3^2$}}
                child{node(32) {$S_3$}
                    child[missing]{}
                    child{node(41) {$C_3$}}
                    child{node(42) {$C_2$}
                        child{node(5) {$1$}}}}}
          child[missing]{}
          child{node(22) {$C_4$}}}
    child{node(12) {$C_3^2: C_4$}
          child[missing]{}
          child[missing]{}
          child{node(23) {$C_4$}}}
    child{node(13) {$C_3^2: C_4$}
          child[missing]{}
          child[missing]{}
          child{node(24) {$C_4$}}}
    child{node(14) {$Q_8$}};

    \draw[gray] (0)--(11) (0)--(12) (0)--(13) (0)--(14);
    \draw[gray] (11)--(21) (11)--(22) (12)--(21) (12)--(23) (13)--(21) (13)--(24) (14)--(22) (14)--(23) (14)--(24);
    \draw[gray] (21)--(31) (21)--(32) (22)--(42) (23)--(42) (24)--(42);
    \draw[gray] (31)--(41) (32)--(41) (32)--(42);
    \draw[gray] (41)--(5) (42)--(5);
\end{tikzpicture}
\end{center}

$\mbox {Max}(U_3(2))= \{Q_8,C_3^2: C_4\}$,

$\mbox {Max}_2(U_3(2))=\{C_4,C_3: S_3\}$,

$\mbox {Max}^\star _2(U_3(2))=\mbox {Max}_2(U_3(2))$.
\end{example}

\begin{example}\label{el217}  
$L_2(17)$ (\cite{gap},\cite[p.~9]{atlas},\cite{cl})

\begin{center}
\begin{tikzpicture}
[level distance=10mm, draw=gray,
level 1/.style={sibling distance=7em, level distance=10mm},
level 2/.style={sibling distance=4em, level distance=12mm},
level 3/.style={sibling distance=4em, level distance=14mm},
level 4/.style={level distance=10mm},
level 5/.style={level distance=10mm},
]
\node {$L_2(17)$}
    child{node(11) {$C_{17}: C_8$}
        child{node(21) {$C_{17}: C_4$}
            child{node(31) {$D_{34}$}
                child[missing]{}
                child{node(41) {$C_{17}$}}  }   }   }
    child{node(12) {$D_{16}$}
        child{node(22) {$C_8$}
            child{node(32) {$C_4$}}    }
        child[missing]{}   }
    child{node(13) {$S_4$}
        child{node(23) {$D_8$}}
        child{node(24) {$A_4$}
            child{node(33) {$C_2^2$} }
            child[missing]{}    }
        child[missing]{}    }
    child{node(14) {$S_4$}
        child{node(25) {$D_8$}}
        child{node(26) {$A_4$}
            child{node(34) {$C_2^2$}
                child{node(42) {$C_2$}
                    child{node(5) {$1$}}    }  }
            child[missing]{}    }
        child{node(27) {$S_3$}} }
    child{node(15) {$D_{18}$}
        child{node(28) {$C_9$}
            child{node(35) {$C_3$}}
            child[missing]{}    }    };

   \draw (11)--(22) (12)--(23) (12)--(25) (13)--(27) (15)--(27);
   \draw (21)--(32) (23)--(32) (25)--(32) (23)--(33) (24)--(35) (25)--(34) (26)--(35) (27)--(35);
   \draw (31)--(42) (32)--(42) (33)--(42) (27)--(42);
   \draw (41)--(5) (42)--(5) (35)--(5);
\end{tikzpicture}
\end{center}

$\mbox {Max}(L_2(17))= \{C_{17}: C_8, S_4, D_{18}, D_{16}\}$,

$\mbox {Max}_2(L_2(17))=\{C_{17}: C_4,C_8,A_4,D_8,S_3,C_9\}=
\mbox {Max}^\star_2(L_2(17))$.
\end{example}

\begin{example}\label{eu33} 

For $U_3(3)$ \cite[p.~9]{atlas} it follows from \cite{gap,cl} that

$\mbox {Max}(U_3(3))= \{(C_3^2: C_3): C_8, (C_4^2: C_3): C_2, SL_2(3): C_4, L_3(2)\}$,

$\mbox {Max}_2(U_3(3))=\{(C_3^2: C_3): C_4,
((C_4\times C_2): C_2): C_3,
C_4^2: C_2, C_4^2: C_3, C_7: C_3, C_3: C_8, S_4\}=\mbox {Max}^\star_2(U_3(3))$.
\end{example}

\section{Some corollaries of the results of Hanguang Meng
and Xiuyun Guo}

The following theorem compiles two results of ~\cite{mg}
with some additions. We give the full proof of all statements
since new information about strictly $2$-maximal subgroups
follows from them.

\begin{theorem}\label{t1}
Let $G$ be a group and $K<G$.
Suppose that there are two maximal subgroups~$M$ and~$H$ of $G$
such that~$K\le M\cap H$, $K$ is maximal in~$M$ and~$K$ is not maximal in~$H$.
The following statements hold.

$(1)$~$K=M\cap H$;

$(2)$~$K_G=M_G<M$;

$(3)$~either $K_G=H_G$ or $KH_G=H$;

$(4)$ if $G$ is soluble, $K\le X\lessdot \,G$
and $K$ is not maximal in~$X$, then $X=H$, $K_G=M_G<H_G$  and~$KH_G=H$.
\end{theorem}

\begin{proof}
(1) Since~$K\le (H\cap M)<M$ and~$K$ is maximal in~$M$, then~$K=H\cap M$.

(2) This statement is Lemma~1~\cite{mg}.
We can assume that $K_G=1$ and $M_G\ne 1$.
Choose a minimal normal subgroup $N$ of $G$
such that $N\le M_G$. Since $K_G=1$,
we get~$N\not \le K$ and~$KN=M$.
In view of~(1), $K=(H\cap M)$,
therefore~$H\cap N\le K$ and~$H\cap N=K\cap N$.
From $K<H\ne M$ and $KN=M$, we conclude that~$N\not \le H$ and~$G=HN$.
The quotient group $KN/N=M/N$ is maximal in $G/N=HN/N$,
consequently,
\[
K/(H\cap N)=K/(K\cap N)\cong  KN/N\lessdot \,G/N=HN/N\cong
H/(H\cap N).
\]
Hence~$K$ is maximal in~$H$, a contradiction, and~$M_G=1=K_G$.
Since~$K<M$, it follows from~$M_G=K_G$ that $M_G\ne M$.

(3) Since $K<H$, we get~$K_G\le H_G$.
Suppose that $K_G<H_G$ and $N/K_G$ is a minimal normal subgroup of $G/K_G$,
$N/K_G\le H_G/K_G$. In view of~(2), $M_G=K_G$, therefore
$N\not \le M$ and~$G=NM$. Since~$NK\le H$ and~$K$ is maximal in $M$,
we conclude that $NK$ is a maximal subgroup of~$G$
and $H=NK=H_GK$.

(4) This statement is Theorem~B~\cite{mg}.
We can assume that $K_G=1$. By the hypothesis~$K\lessdot \,M\lessdot \,G$.
Suppose that there are two subgroups $H\in \mbox {Max} (G,K)$
and~$X\in \mbox {Max} (G,K)$ such that~$K\le H\cap X$, $H\ne X$,
$K$ is not maximal in~$H$ and~$K$ is not maximal in~$X$.
In view of~(2), $M_G=1$, $K=M\cap H=M\cap X$.
Since~$G$ is a soluble primitive group, we obtain
\[
G=N: M, \ N=F(G), \ \Phi  (G)=1,
\]
$N$ is the unique minimal normal subgroup of~$G$.
If $N\le H\cap X$, then
\[
H=H\cap (NM)=N(H\cap M)=N(X\cap M)=X\cap (NM)=X,
\]
a contradiction. Hence~$N\not \le H$ or~$N\not \le X$.
Suppose that~$N\not \le H$ to fix an idea.
So~$G=N: H$. From~$K\lessdot \,M\lessdot \,G$,
it follows that~$NK\lessdot \,G$.
By the hypothesis, $K$ is not maximal in~$H$,
therefore there is a subgroup~$T$ such that~$K<T\lessdot \,H$.
Now $NK<NT\lessdot \,G$, a contradiction.
Hence we conclude that~$H=X$.

If $K_G=H_G$, then~$K$ and~$H$ are conjugate.
Since~$K_G=M_G$ and~$K/M_G\lessdot \,M/M_G$,
we get~$K/H_G\lessdot \,H/H_G$,
and~$K$ is maximal in~$H$, a contradiction.
Therefore $K_G\ne H_G$, and it follows from~(3)
that $KH_G=H$.
\end{proof}

\begin{corollary}\label{c1}
Let $G$ be a group,  $K\lessdot \,M\lessdot \,G$.
If~$K_G\ne M_G$, then~$K$ is a strictly $2$-maximal subgroup of $G$.
In particular, if a maximal subgroup~$M$ of~$G$ is normal in~$G$,
then every maximal subgroup of~$M$ is a strictly $2$-maximal subgroup of~$G$.
\end{corollary}

\begin{note}
Let~$G$ be a soluble group and let $M$ be a maximal subgroup in $G$
of least index. According to~\textup{\cite[Lemma 1]{mt}},
$M$ is normal in~$G$, and in view of Corollary~\ref{c1},
all maximal subgroups of~$M$ are strictly $2$-maximal subgroups of~$G$.
In insoluble groups, it is not true.
\end{note}

\begin{example}\label{ea6}

In~$A_6$ (\cite[IdGroup(360,118)]{gap}, \cite{GN}),
a maximal subgroup~$A_5$ has the least index
and $S_3$ is a maximal subgroup of ~$A_5$.
Since
\[
S_3\lessdot C_3: S_3\lessdot C_3^2: C_4\lessdot G,
\]
$S_3$ is a $3$-maximal subgroup of~$A_6$.
Hence~$S_3$ is not a strictly $2$-maximal subgroup of~$A_6$.
\end{example}


\begin{corollary}\label{c2}
Let $G$ be a soluble group and let~$K$ be a $2$-maximal subgroup of $G$.
If $K$ is not a strictly $2$-maximal subgroup of~$G$,
then there is a unique maximal subgroup~$V$ of~$G$
such that~$K<V$ and~$K$ is not maximal in~$V$.
\end{corollary}

\begin{example}\label{el227}
In $L_2(3^3)$ (\cite[p.~18]{atlas},\cite{cl}), there is a maximal subgroup~$M\cong D_{26}$.
A subgroup~$K$ of order~2 from $M$ is a $2$-maximal subgroup
of~$L_2(3^3)$. In~$L_2(3^3)$, there are maximal subgroups~$H\cong D_{28}$
and~$U\cong A_4$. Since Sylow $2$-subgroups of~$L_2(3^3)$ are of order~4
and conjugate, we can assume that~$K\le H\cap U$.
As~$K$ is $2$-maximal in~$H$ and in~$U$, we have~$K$ is
a $3$-maximal subgroup of~$L_2(3^3)$.
Consequently, the condition of group solubility
could not be removed in Corollary~\ref{c2}.
\end{example}

\begin{lemma}\label{l4}
Let $H$ be a subgroup of a $p$-soluble group $G$ and~$|G:H|=p$.
Then $G/H_G$ is supersoluble.
\end{lemma}

\begin{proof}
We can assume $H_G=1$. Since $O_{p^\prime}(G)\le H_G=1$,
we obtain that $O_{p}(G)\ne 1$ and $C_G(O_{p}(G))\le O_{p}(G)$.
Since $H\cap O_{p}(G)\le H_G=1$, we have $G=O_p(G): H$,
$|O_{p}(G)|=p$ and $H$ is isomorphic to a subgroup
of a cyclic group of order $p-1$. Therefore $G$ is supersoluble.
\end{proof}

\begin{corollary}\label{p2}
Let $G$ be a $p$-soluble group, $M<G$. If $|G:M|=p$,
then every maximal subgroup in~$M$ is a strictly
$2$-maximal subgroup of $G$.  In particular,
in a soluble group, all maximal subgroups of
a subgroup of prime index are
strictly $2$-maximal subgroups of a group.
\end{corollary}

\begin{proof}
Assume that the assertion is false.
Then there are a maximal subgroup~$K$ in $M$
and a maximal subgroup~$H$ in $G$ such that
$K<H$ and~$K$ is not maximal in~$H$.
According to Theorem~\ref{t1},
$M$ is not normal in~$G$ and~$K_G=M_G\le H_G$.
The quotient group~$G/K_G$ is supersoluble
by Lemma~\ref{l4}, therefore~$K$ is
a strictly $2$-maximal subgroup of~$G$,
a contradiction.
\end{proof}

\begin{note}
We do not know whether the requirement of group $p$-solubility
could be removed in Corollary~\ref{p2}.
\end{note}

\begin{corollary} 
Let $G$ be a group and $K\le G$. Suppose that there are
two maximal subgroups~$M$ and~$H$ in $G$ such that
$K$ is maximal in~$M$ and~$K$ is not maximal in~$H$.
If $K$ is subnormal in~$G$, then~$K=M_G$ and $G/K=H/K: M/K$
is a nonprimary nonsupersoluble group in which all proper subgroups
are primary.
\end{corollary} 

\begin{proof}
According to Theorem~\ref{t1}, $M$ is not normal in~$G$
and~$K_G=M_G\le H_G$. By the hypothesis,
$K$ is subnormal in~$G$, therefore $K$ is normal in~$G$ and~$K=M_G$.
Since~$K$ is maximal in~$M$, we get $|M/K|=p$ for a prime~$p$.
As~$N_{G/K}(M/K)=M/K$ we deduce $G/K$ is soluble~\cite[IV.7.4]{hup},
$H/K$ is normal in~$G/K$~\cite[II.3.2]{hup}
and~$H/K$ is a minimal normal subgroup.
Hence $H/K$ is a $q$-group for a prime $q\ne p$.
Since~$K$ is not maximal in~$H$, we have~$|H/K|>q$,
and~$G/K=H/K: M/K$ is a nonprimary nonsupersoluble group.
It is clear that all proper subgroups in~$G/K$ are primary.
\end{proof}


\begin{thebibliography}{3}

\parskip=-3pt

\bibitem{gap}
\textit{The GAP Group: GAP --- Groups, Algorithms, and Programming}.
Ver. 4.11.0 released on 29 February 2020.
http://www.gap-system.org.

\bibitem{mkRic}
\textsc{V.\,S. Monakhov, V.\,N. Kniahina.}
Finite group with $\mathbb P$-subnormal subgroups.
\textit{Ricerche Mat.}, 62 (2013),  307--323.

\bibitem{kour}
\textit{Unsolved Problems in Group Theory. The Kourovka Notebook}.
Institute of Mathematics SO RAN, Novosibirsk 19 (2018),
https://arxiv.org/pdf/1401.0300.pdf.

\bibitem{mg}
\textsc{H. Meng, X. Guo.}
    Weak second Maximal subgroups in solvable groups.
    \textit{J. Algebra}, 517 (2019), 112--118.

\bibitem{hup}
\textsc{B. Huppert}.
\textit{Endliche Gruppen I}.
Berlin, Springer (1967).
	
\bibitem{atlas}
\textsc{J.\,H. Conway, R.\,T. Curtis, S.\,P. Norton, R.\,A.~Parker, R.\,A.~Wilson}.
\textit{Atlas of Finite Groups:
Maximal Subgroups and Ordinary Characters for Simple Groups}.
Oxford, Clarendon Press (1985).

\bibitem{cl}
\textsc{T. Connor, D. Leemans}.
An atlas of subgroup lattices of finite almost simple groups.
\textit{Ars Math. Contemp.}, 8 (2015), 259--266.

\bibitem{GN}
\textsc{T. Dokchitser}.
\textit{GroupNames},
http://groupnames.org/.

\bibitem{mt}
\textsc{V. S. Monakhov, V. N. Tyutyanov}.
Finite groups with supersoluble subgroups of given orders.
\textit{Trudy Inst. Mat. i Mekh. UrO RAN}, 25(4) (2019), 155--163 (In Russian).

\end{thebibliography}
\end{document}